\documentclass[12pt,a4paper]{amsart}
\usepackage{tikz}
\usepackage{pdflscape}
\usepackage{amsmath}
\usepackage{amsfonts}
\usepackage{amssymb}
\usepackage{mathtools}
\usepackage[all]{xy}
\usepackage{color}
\usetikzlibrary{arrows}
\usepackage{float}
\usepackage[shortlabels]{enumitem}
\usepackage{ifpdf}
\ifpdf
  \usepackage[colorlinks,
            linkcolor=magenta,       
            anchorcolor=magenta,     
            citecolor=magenta,       
            final,
            hyperindex]{hyperref}
\else
  \usepackage[colorlinks,
            linkcolor=magenta,       
            anchorcolor=magenta,     
            citecolor=magenta,       
            final,
            hyperindex,driverfallback=hypertex]{hyperref}
\fi
\newtheorem{theorem}{Theorem}[section]
\newtheorem{lemma}[theorem]{Lemma}

\newtheorem{problem}[theorem]{Problem}

\newtheorem{prop}[theorem]{Proposition}
\theoremstyle{definition}
\newtheorem{defn}[theorem]{Definition}
\newtheorem{remark}[theorem]{Remark}

\makeatletter \@addtoreset{equation}{section} \makeatother
\allowdisplaybreaks

\newcommand{\delete}[1]{}

\newcommand{\phii}[2]{\phi_{#1,#2}}
\begin{document}
\title[Dual weak left $\star$-braces]{Deformed solutions of the Yang-Baxter equation associated to dual weak left $\star$-braces}

\author{Shoufeng Wang}
\address{School of Mathematics, Yunnan Normal University, Kunming, Yunnan 650500, China}
\email{wsf1004@163.com}


\begin{abstract}As generalizations of dual weak left braces and skew left braces, in this paper, dual weak left $\star$-braces and square skew left braces are introduced, respectively. We firstly show that a dual weak left $\star$-brace is exactly a strong semilattice of a family of square skew left braces. Then we introduce  distributors for dual weak left $\star$-braces and prove that the map deformed by each  distributor  is always a solution of the Yang-Baxter equation.  Our work
may be regarded as extending and enriching some related results on skew left braces and weak left braces in literature.
\end{abstract}
\makeatletter
\@namedef{subjclassname@2020}{\textup{2020} Mathematics Subject Classification}
\makeatother

\maketitle



\vspace{-0.8cm}

\section{Introduction}
The quantum Yang-Baxter equation first appeared in theoretical physics in Yang \cite{Yang} and  in statistical mechanics  in Baxter \cite{Baxter} independently. Now,  the Yang-Baxter equation has many applications in the areas of mathematics and mathematical physics such as knot theory, quantum computation, quantum group theory and so on. Let $V$ be a vector space and $R: V\otimes V \rightarrow V\otimes V$ be a linear transformation. Then $R$ is called  {\em a solution on a vector space} of the Yang-Baxter equation if $R_{12}R_{23}R_{12} =R_{23}R_{12}R_{23} \mbox{ in End}(V\otimes V\otimes V),$
where $R_{ij}$ acts as $R$ on the $i$-th and $j$-th components, and as the identity on the remaining component.

Finding all solutions of the Yang-Baxter equation is presently impossible, so in 1992, Drinfeld \cite{Drinfeld} posed the question of finding all the so called set-theoretic  solutions of the Yang-Baxter equation as they form solutions on vector spaces by linearly extending them. Let $S$ be a non-empty set and $r: S\times S\longrightarrow S\times S$ be a map. Then for any $a\in S$, $r$ induces the following maps:
$$\lambda_{a}: S\rightarrow S, \,b\mapsto \mbox{the first component of }r(a,b),$$$$
\rho_{b}: S\rightarrow S,\, a\mapsto \mbox{the second component of }r(a,b).$$
Thus, for all $a,b\in S$, we have $r(a,b)=(\lambda_{a}(b),\rho_{b}(a))$.
If $r$ satisfies the following equality $$(r\times {\rm id}_{S})({\rm id}_{S}\times r)(r\times {\rm id}_{S})=({\rm id}_{S}\times r)(r\times
{\rm id}_{S})({\rm id}_{S}\times r)$$ in the set of maps from $S\times S\times S$ to $S\times S\times S$, where  ${\rm id}_{S}$ is the identity map on $S$, then  $r_S$ is called a {\em set-theoretic  solution} of the Yang-Baxter equation, or briefly a {\em solution}.
It is a routine matter to prove  that $r$ is  a   solution  of the Yang-Baxter equation if and only if
\begin{equation}\label{jie1}
\lambda_{x}\lambda_{y}(z)=\lambda_{\lambda_{x}(y)}\lambda_{\rho_{y}(x)}(z),
\end{equation}
\begin{equation}\label{jie2}
\rho_{z}\rho_{y}(x)=\rho_{\rho_{z}(y)}\rho_{\lambda_{y}(z)}(x),
\end{equation}
\begin{equation}\label{jie3}
\lambda_{\rho_{\lambda_{y}(z)}(x)}\rho_{z}(y)=\rho_{\lambda_{\rho_{y}(x)}(z)}\lambda_{x}(y)
\end{equation}
for all $x,y,z\in S$. Moreover, a solution $r$ is called {\em left non-degenerate} (respectively, right non-degenerate) if $\lambda_a$ (respectively, $\rho_a$) is bijective for each $a\in S$. A {\em non-degenerate solution} is a solution which is both left and right non-degenerate.    A solution
that is neither left nor right non-degenerate is called {\em degenerate}.

Set-theoretic solutions of the Yang-Baxter equation are investigated extensively in recent years, see for example
\cite{Etingof-Schedler-Soloviev,Gateva-Ivanova-Van den Bergh,Lu-Yan-Zhu,Soloviev}. Subsequently, non-degenerate involutive solutions have been studied by many authors, see the papers \cite{Cedo,Cedo-Jespers-Okninski,Rump1,Rump2} and the  references therein. In particular,   Rump \cite{Rump1,Rump2} introduced left cycle sets and left braces to investigate non-degenerate involutive solutions, respectively.  In 2007, Rump introduced left braces in \cite{Rump2} as a generalization of Jacobson radical rings, and a few years later, Ced$\acute{\rm o}$, Jespers, and Okni$\acute{\rm n}$ski \cite{Cedo-Jespers-Okninski} reformulated
Rump's definition of left braces.  Since then, left braces have become the most used tool in the investigations of  non-degenerate involutive solutions of the Yang-Baxter equation. In order to study non-degenerate solutions that are not necessarily involutive, Guarnieri and Vendramin \cite{Guarnieri-Vendramin} introduced a generalization of left braces, namely, skew left braces. To state the notions of skew left braces and left braces, we need recall some necessary terminologies and notations.

Let $(S, \cdot)$ be a semigroup.  As usual, for all $x, y\in S$, we denote $x\cdot y$ by $xy$.  Let $e\in S$. Then $e$ is called idempotent if $e^2=e$. As usual, we denote the set of idempotent elements in $(S, \cdot)$ by $E(S, \cdot)$.
Recall  from Howie \cite{Howie} that a semigroup $(S, \cdot)$ is called an {\em inverse semigroup} if, for every $a\in S$, there exists a unique element $b$ in  $S$ such that $aba=a$ and $bab=b$. We denote such an element $b$ by $a^{-1}$ and call it the {\em inverse } of $a$ in $(S, \cdot)$. In this case, we have
\begin{equation}\label{inverse semigroup}
(a^{-1})^{-1}=a^{-1} \mbox{ and } (ab)^{-1}=b^{-1}a^{-1}\mbox{ for all } a,b\in S.
\end{equation}
  Obviously, groups are inverse semigroups, and if $(S, \cdot)$ is a group and  $a\in S$, then $a^{-1}$ above is exactly the usual inverse of $a$ in the group $(S, \cdot)$.
Let $(S, +)$ and $(S, \cdot)$  be two inverse semigroups.  Then we usually denote the inverse  of $x$ in $(S, +)$ by $-x$ and denote $x+(-y)$ by $x-y$ for all $x, y\in S$. {\bf To avoid parentheses,  throughout this paper we will assume that the multiplication has higher precedence than the addition}. For  example, we use $xy^{-1}-z+y+zw$ to denote $(x\cdot (y^{-1}))+(-z)+y+(z\cdot w)$  for all $x, y, z, w\in S$.

Let $(S, +)$ and $(S, \cdot)$ be two groups. The triple $(S, +, \cdot)$ is called a {\em skew left brace}  if the following axiom holds:
\begin{equation}\label{skew brace}x(y+z)=xy-x+xz.
\end{equation}
If this is the case, the identities in $(S, +)$ and $(S, \cdot)$ coincide.  A skew left brace $(S, +, \cdot)$ is called a {\em left brace} if $(S, +)$  is an abelian group.
In \cite[Theorem 3.1]{Guarnieri-Vendramin}, Guarnieri and Vendramin have proved that every skew left brace $(S, +, \cdot)$ can give rise to some bijective non-degenerate solution $r_S$. One can prove that $r_S$ can be rewritten as
\begin{equation}\label{associated map}
\begin{array}{cc}
r_S: S\times S \rightarrow  S\times S,\,\,\,\,  (x, y)\mapsto (x(x^{-1}+y), (x^{-1}+y)^{-1}y),\\[2mm]
\mbox{i.e. } r_S(x, y)=(x(x^{-1}+y), (x^{-1}+y)^{-1}y) \mbox{ for all } x, y\in S,
\end{array}
\end{equation}
where $x^{-1}$ is the inverse of $x$ in the group $(S, \cdot)$ for all $x\in S$. Moreover, $r_S$ is involutive if and only if  $(S, +, \cdot)$ is a left brace. The above $r_S$ is called {\em the map associated to $(S, +, \cdot)$} in literature.
More details on left braces and skew left braces, the readers can  consult
the survey articles \cite{Cedo,Vendramin}. More recently,  Doikou and Rybo{\l}owicz \cite{DoRy22x}  presented a way to assign a new ``deformed" solution to particular elements of skew left braces. In the case of the identity element, they get the usual solution  associated to a skew left brace.

On the other hand, in 2022, Catino, Mazzotta, Miccoli and Stefanelli \cite{Catino-Mazzotta-Miccoli-Stefanelli} have introduced {\em weak left braces} by using so-called {\em Clifford semigroups} to study not necessarily bijective solutions. In fact, a similar approach of weakening the structure was already considered for the quantum Yang-Baxter equations by introducing weak Hopf algebras in Li \cite{Li}). A triple $(S, +, \cdot)$ is called a {\em weak left brace} if $(S, +)$ and $(S, \cdot)$ are  inverse  semigroups and the following axioms hold:
\begin{equation}\label{weak semibrace}
x(y+z)=xy-x+xz,\,\,\,\,\,\, xx^{-1}=-x+x.
\end{equation}
In this case, $r_S$ in (\ref{associated map}) is well defined, but $x^{-1}$ is the inverse of $x$ in the inverse semigroup $(S, \cdot)$ for all $x\in S$. According to \cite[Theorem 11]{Catino-Mazzotta-Miccoli-Stefanelli},  the map  given in (\ref{associated map})  associated to a weak left  brace is always  a solution. More recent results on  weak left braces can be found in \cite{Catino-Colazzo-Stefanelli3,Catino-Mazzotta-Stefanelli2,Catino-Mazzotta-Stefanelli3,Gong-Wang, Mazzotta-Rybolowicz-Stefanelli}.
In 2023,  Catino, Mazzotta and Stefanelli \cite{Catino-Mazzotta-Stefanelli2} introduced a class of weak left braces, namely, {\em dual weak left braces}, and established the relationship between dual weak left braces and Rota-Baxter operators on Clifford semigroups. In 2024, among other things, Catino, Mazzotta and Stefanelli \cite{Catino-Mazzotta-Stefanelli3}
determined the algebraic structures of dual weak left braces. More specifically, they proved that a dual weak left brace is a strong semilattice of a family of skew left braces.  In 2025, Mazzotta, Rybo{\l}owicz and Stefanelli \cite{Mazzotta-Rybolowicz-Stefanelli} generalized some results in  \cite{DoRy22x} on deformed solutions of Yang-Baxter equation on skew left braces to deformed solutions  on weak left braces. In 2026, Gong and Wang \cite{Gong-Wang} introduced post Clifford semigroups and connected dual weak left brace with post Clifford semigroups and other related algebraic structures.

 As a generalization of inverse semigroups, regular $\star$-semigroups  were introduced by   Nordahl  and Scheiblich in \cite{Nordahl-Scheiblich} in 1978.
Let $(S, \cdot)$ be a semigroup  and $\ast: S\rightarrow S,\, x\mapsto x^\ast$ be a map.  Recall from \cite{Nordahl-Scheiblich} that the triple  $(S, \cdot,\, \ast)$ is called {\em a regular ${\star}$-semigroup} if the following axioms are satisfied:
\begin{equation}\label{shishi1.2} x x^\ast x=x,\,\, x^{\ast\ast}=x,\,\,
(xy)^\ast=y^\ast x^\ast,
\end{equation} where $x^{\ast\ast}=(x^\ast)^\ast$. We denote $$P(S, \cdot)=\{e\in E(S, \cdot)\mid e^\ast=e\}$$ and called it {\em the set of projections of $(S,\cdot, \ast)$}. Observe that the axiom $x^\ast x x^\ast= x^\ast$ also holds in this case. Obviously, if $(S, \cdot)$ is an inverse semigroup, then by the definition of inverse semigroups and (\ref{inverse semigroup}), $(S, \cdot, \ast)$ forms a regular $\star$-semigroup by setting $x^\ast=x^{-1}$ for all $x\in S$.

In 2025, as generalizations of weak left braces,  Liu and Wang \cite{Liu-Wang} have introduced weak left $\star$-braces by using regular $\star$-semigroups and determined the algebraic structures of weak left $\star$-braces. Moreover, they have also proved that the associated maps to weak left $\star$-braces are always  solutions of the Yang-Baxter equation. This work extended the results on weak left braces in \cite{Catino-Mazzotta-Miccoli-Stefanelli}.

From the above statements, the following question is natural:
\begin{problem}
How to establish a theory parallel to the theory in \cite{Catino-Mazzotta-Stefanelli3,Mazzotta-Rybolowicz-Stefanelli} in the class of weak left $\star$-braces?
\end{problem}
In this paper, we shall consider the above problem. Section 2 contains some preliminaries on skew left braces, regular $\star$-semigroups and weak left $\star$-braces, respectively. In Section 2, we introduce dual weak left $\star$-braces and obtain some structural properties of them, which generalizes the corresponding results on  weak left braces given in  Catino, Mazzotta and Stefanelli \cite{Catino-Mazzotta-Stefanelli3}.  In Section 4, we introduce {\em distributors} of dual weak left $\star$-braces and prove that the maps deformed by each distributor is always a solution of the Yang-Baxter equation,
which generalizes the corresponding results on  weak left braces given in   Mazzotta, Rybo{\l}owicz and Stefanelli \cite{Mazzotta-Rybolowicz-Stefanelli}.
\section{Preliminaries}
In this section,  we recall and give some notions and basic facts on regular $\star$-semigroups, skew left braces and weak left $\star$-braces.
Recall from Howie \cite{Howie} that the {\em Green's relations} ${\mathcal L}$ and ${\mathcal R}$ on a semigroup $(S, \cdot)$ are defined as follows: For all $a,b\in S$, $$a{\mathcal L} b \mbox{ if and only if } a=b \mbox{ or  there exist } u, v\in S \mbox{ such that }a=ub,\, \, b=va,$$
$$a{\mathcal R} b \mbox{ if and only if } a=b \mbox{ or  there exist  } u, v\in S \mbox{ such that }a=bu,\, \, b=av,$$
$$a{\mathcal J} b \mbox{ if and only if } a=b \mbox{ or  there exist  } u, v, s,t\in S \mbox{ such that }a=ubv,\, \, b=sat.$$
Observe that both $\mathcal L$ and $\mathcal R$ are  equivalences on $S$. Let $${\mathcal L}\circ {\mathcal R}=\{(x,y)\in S\times S\mid \mbox{there exists } z\in S \mbox{ such that } x {\mathcal L} z {\mathcal R} y\},$$$${\mathcal R}\circ {\mathcal L}=\{(x,y)\in S\times S\mid \mbox{there exists } z\in S \mbox{ such that } x {\mathcal R} z {\mathcal L} y\}.$$
By \cite[Proposition 2.1.3]{Howie}, we have ${\mathcal L}\circ {\mathcal R}={\mathcal R}\circ {\mathcal L}$.
Denote ${\mathcal D}={\mathcal L}\circ {\mathcal R}$. Then $\mathcal D$ is the smallest equivalence on $S$ which containing $\mathcal L$ and $\mathcal R$.
The following lemma  collects some  basic properties of regular $\star$-semigroups.
\begin{lemma}[\cite{Nordahl-Scheiblich,Yamada2}] \label{1.2} Let $(S,\cdot,\, \ast)$ be a regular $\star$-semigroup, $e,f\in P(S,\cdot)$ and $a, b\in S$.
\begin{itemize}
\item[(1)] $P(S,\cdot)=\{ xx^\ast\mid x\in S\} =\{x^\ast x\mid x\in S\}$ and $e^\ast=e\in E(S, \cdot)$.
\item[(2)] $E(S,\cdot)=(P(S,\cdot))^2=\{ef\mid e,f \in P(S, \cdot)\}$.
\item[(3)] $a{\mathcal R} b$ (respectively, $a{\mathcal L} b$) if and only if  $ aa^\ast =bb^\ast $ (respectively, $a^\ast a =b^\ast b$).
\item[(4)] $aea^\ast, fef\in P(S,\cdot)$ and $a {\mathcal D} a^\ast$.
\end{itemize}
\end{lemma}

Let $(S, \cdot)$ be a semigroup. Then $(S, \cdot)$ is called {\em regular} if for every $x\in S$, there exists $a\in S$ such that $axa=a$ and $xax=x$.
Obviously, inverse semigroups and regular $\star$-semigroups are regular semigroups.
 Moreover,  recall from Howie \cite{Howie} that  a regular semigroup  $(S, \cdot)$ is called
\begin{itemize}
\item{\em simple} if ${\mathcal J}=S\times S$, (\cite[Section 3.1]{Howie})
\item  {\em orthodox} if  $E(S,\cdot)$ is a  subsemigroup of $(S,\cdot)$,\,\, (\cite[Section 6.2]{Howie})
\item  {\em completely regular}  if $a\, {\mathcal H}\, a^2$ for all $a\in S$,  (\cite[Theorem 2.2.5 and Proposition 4.1.1]{Howie})
\item{\em completely simple} if $S$ is both simple and completely regular,  (\cite[Theorem 3.3.2]{Howie})
\item {\em locally inverse} if $(eSe, \cdot)$ is an inverse  semigroup  for each $e\in E(S,\cdot)$.   (\cite[Section 6.1]{Howie})
\end{itemize}
\begin{lemma}[Theorem 4.1 in \cite{Petrich1}]\label{wanquandan1}Let $(S,\cdot,\,  \ast)$ be a regular $\star$-semigroup.  Then $(S,\cdot,\,  \ast)$ is   completely simple if and only if $xx^\ast=x yy^\ast x^\ast$ for all $x\in S$.
\end{lemma}
\begin{lemma}[Theorem 4.1 in \cite{Auinger}]\label{jubuni}Let $(S,\cdot,\,  \ast)$ be a regular $\star$-semigroup.  Then $(S,\cdot,\,  \ast)$ is locally inverse  if and only if $xyy^\ast x^\ast xzz^\ast x^\ast=xzz^\ast x^\ast xyy^\ast x^\ast$ for all $x, y, z\in S$.
\end{lemma}
\begin{lemma}[Lemma 2.7 in \cite{Liu-Wang}]\label{zicu8}Let $(S,\cdot,\,  \ast)$ be a regular $\star$-semigroup. Then
$(S,\cdot,\,  \ast)$ is orthodox and locally inverse if and only if $afgb=agfb$ for all  $a,b\in S$ and $f,g\in P(S, \cdot)$.
\end{lemma}

\begin{lemma}\label{wanquandan}
Let $(S,\cdot,\,  \ast)$ be a regular $\star$-semigroup which is completely simple and orthodox. Then $afgb=ab$  and $fgf=f$for all $f,g\in P(S)$ and $a,b\in S$.
\end{lemma}
\begin{proof}
Let $a,b\in S$ and $f,g\in P(S,\cdot)$. By hypothesis and  Lemmas \ref{wanquandan1} and \ref{jubuni}, it is easy to see that $S$  is locally inverse.
So by Lemmas \ref{1.2}, \ref{wanquandan1} and \ref{zicu8}, we have
$$afgb=a(a^\ast a fgb)=a(a^\ast a f a^\ast a g a^\ast ab)=aa^\ast (afa^\ast)(aga^\ast)ab$$$$=aa^\ast (aff^\ast a^\ast)(agg^\ast a^\ast)ab=aa^\ast aa^\ast aa^\ast ab=ab.$$
In particular, if $a=f=b$, then we have $fgf=ffgf=afgb=ab=ff=f.$
\end{proof}

\begin{lemma}[Lemma 2.8 in \cite{Liu-Wang}]\label{zicu9}Let $(S,\cdot,\,  \ast)$ be a regular $\star$-semigroup.  Then $(S,\cdot,\,  \ast)$ is completely regular, orthodox and locally inverse if and only if $x y=x y^{\ast}yy$ for all $x, y\in S$.
\end{lemma}


Assume that $(S,+,\, -)$ and $(S,\cdot,\, \ast)$ are regular $\star$-semigroups, where $-: S\rightarrow S,\, x\mapsto -x$ and $\ast: S\rightarrow S,\, x\mapsto x^\ast$ are two unary operations on $S$. {\bf To avoid parentheses,  throughout this paper we always  assume that the multiplication has higher precedence than the addition and write $x+(-y)$ as $x-y$ for all $x, y\in S$.}

\begin{defn}[Definition 3.1 in \cite{Liu-Wang}] Let $(S,+,\, -)$ and $(S,\cdot,\, \ast)$ be regular $\star$-semigroups. Then $(S,+,\cdot,\, -,\, \ast)$ is called a {\em weak  left $\star$-brace} if the following axioms hold:
\begin{equation}\label{lw5}x(y+z)=xy-x+xz,\,\,\, -x+x=xx^\ast.
\end{equation}
\end{defn}
Let  $(S,+,\cdot,\, -,\, \ast)$ be a weak  left $\star$-brace. By Lemma \ref{1.2} (1) and the second identity in (\ref{lw5}),  the sets of projections $P(S,+)$ and $P(S, \cdot)$ of $(S, +, -)$ and $(S, \cdot, \ast)$ coincide. That is,
\begin{equation}\label{toushejiyizhi}
\begin{array}{cc}
P(S,+)=\{e\in E(S, +)\mid -e=e\}=\{x-x\mid x\in S\}=\{-x+x\mid x\in S\}\\[2mm]
=\{xx^\ast\mid x\in S\}=\{x^\ast x\mid x\in S\}=\{f\in E(S, \cdot)\mid f^\ast=f \}=P(S, \cdot).
\end{array}
\end{equation}
We  denote it by $P(S)$ and call  it {\em the set of projections of the  weak  left $\star$-brace $(S,+,\cdot,\, -,\, \ast)$}.

\begin{lemma}[Lemma 3.6 in \cite{Liu-Wang}]\label{lw6} Let $(S,+,\cdot,-,\ast)$ be a  weak left $\star$-brace,  $a,b\in S$ and $e\in P(S)$.
\begin{itemize}
\item[(1)]   $a^{\ast\ast}=a,\, aa^\ast a=a,\, (ab)^\ast=b^\ast a^\ast, a^\ast aa^\ast=a^\ast.$
\item[(2)] $-(-a)=a, a-a+a=a,\, -(a+b)=-b-a,\, -a+a-a=-a.$
\item[(3)] $e^\ast=-e=ee=e+e=e-e=ee^\ast=e^\ast e=-e+e=e$.
\end{itemize}
\end{lemma}
By \cite[Lemmas 3.12, 3.13,  4.1 and 4.2]{Liu-Wang} and Lemma \ref{lw6}, we have the following result.
\begin{lemma}\label{liwei} Let $(S,+,\cdot,-,\ast)$ be a weak left $\star$-brace, $x, y\in S$ and $e\in P(S)$.  Then
$$ -x+xe=-xe+x,\,ex=x+e,\, x(y-y)y=xy, e+x+y=e+x+e+y.$$
\end{lemma}
\begin{lemma}[Theorem 3.16 in \cite{Liu-Wang}]\label{zicu9tiaoxingzhi}Let $(S,+,\cdot,-,\ast)$ be a weak left $\star$-brace. Then $(S, +, -)$ is completely regular, orthodox and locally inverse, and $(S, \cdot, \ast)$ is orthodox and locally inverse.
\end{lemma}
To end this section, we recall some results on skew left braces which will be needed in our paper.  Let $(G,+, \cdot)$  be a skew left brace. From \cite{Mazzotta-Rybolowicz-Stefanelli}, one call
\begin{align*}
    \mathcal{D}_r(G)=\{t \in S  \mid (x+y)t=xt-t+yt \mbox{ for all }x,y \in G\}
\end{align*}
the set of \emph{right distributors} of $(G,+, \cdot)$.
\begin{lemma}[Lemma 3.3 in \cite{Mazzotta-Rybolowicz-Stefanelli}]\label{ma}Let $(G, +, \cdot)$ be a skew left brace and $z\in G$. Then
$z\in \mathcal{D}_r(G)$ if and only if $(a-b+c)z=az-bz+cz$ for all $a,b,c\in G$.

\end{lemma}

Let $(G,+, \cdot)$  be a skew left brace and $x,y,t\in G$. Define four maps $\check{\sigma}_x^t, \check{\tau}_y^t, \hat{\sigma}_x^t, \hat{\tau}_y^t$ from $G$ to $G$  by  setting  $$\check{\sigma}^t_x(y)=xy-xt + t,\,\,\,\, \check{\tau}^t_y(x)=(\check{\sigma}^t_x(y))^{-1}xy=\left(xy-xt + t\right)^{-1}  xy,$$
$$\hat{\sigma}^t_x(y)=-xt + xyt,\,\,\,\, \hat{\tau}^t_y(x)=(\hat{\sigma}^t_x(y))^{-1}xy=\left(-xt + xyt\right)^{-1}  xy,$$
for all $x,y \in G$. Define two maps from $G\times G$ to $G\times G$ by setting
$$\check{r}^G_t(x,y)=\left(\check{\sigma}^t_x(y), \check{\tau}^t_y(x)\right),\,\, \hat{r}^G_t(x,y)=\left(\hat{\sigma}^t_x(y), \hat{\tau}^t_y(x)\right)$$
for all $x, y \in G$.

\begin{lemma}[\cite{DoRy22x,Mazzotta-Rybolowicz-Stefanelli}]\label{zhuyin1}
Let $(G,+, \cdot)$  be a skew left brace with identity $1$ and $x,y,t\in G$.
\begin{itemize}
\item[(1)] $\mathcal{D}_r(G)$ is subgroup of $(G, \cdot)\,\,\,\,\, (\text{\cite[Proposition 2.5]{DoRy22x}}).$
\item[(2)] If $t\in \mathcal{D}_r(G)$, then $t^{-1}\in \mathcal{D}_r(G)$ and $\check{r}^G_{t^{-1}}$ is a non-degenerate and bijective solution of the Yang-Baxter equation, whose inverse map is $\hat{r}^G_{t}$, and $(\hat{\sigma}^t_x)^{-1}=\hat{\sigma}^{t^{-1}}_{x^{-1}}$, $(\hat{\tau}^t_x)^{-1}=\hat{\sigma}^{t}_{x^{-1}}$.
    $($\cite[Theorem 2.6, Proposition 2.12]{DoRy22x}, \cite[Remark 3.2]{Mazzotta-Rybolowicz-Stefanelli}$)$
\item[(3)] $\hat{r}^G_{t}$ is  a solution of the Yang-Baxter equation if and only if $t\in \mathcal{D}_r(G)$. In this case,
$\hat{\tau}^t: (S,\cdot)\rightarrow {\rm{Map}}(S), x\mapsto \hat{\tau}^t_x$ is an anti-homomorphism, and $\hat{\sigma}^t: (S,\cdot)\rightarrow {\rm{Map}}(S), x\mapsto \hat{\sigma}^t_x$ is a homomorphism if and only if $wt=t+w$ for all $w\in G$.  $($\cite[Theorem 3.9, Lemma 3.7 (2) and Proposition 3.11]{Mazzotta-Rybolowicz-Stefanelli}$)$
\item[(4)] $\hat{\sigma}^t_x   \hat{\sigma}^t_1  (x^{-1}(xt+y)t^{-1})=-xt+(x+y)t$ and $\hat{\sigma}^t_1  \hat{\sigma}^t_x   (x^{-1}(xt+y)t^{-1})=-t+yt$.
 $(${\rm cf.} \cite[the fina part of the proof of Theorem 3.9]{Mazzotta-Rybolowicz-Stefanelli}$)$
\end{itemize}

\end{lemma}

\section{Dual weak left $\star$-braces}
In this section, we introduce the notions of dual weak left $\star$-braces and square skew left braces, respectively,  and establish their algebraic structures.
\begin{defn}\label{xxx2}
A  weak left $\star$-brace $(S,+,\cdot,-,\ast)$ is called a {\em dual weak left $\star$-brace} if $x-x=x^\ast x \mbox{ for all } x\in S.$
\end{defn}

\begin{lemma}\label{likang5}
Let $(S,+,\cdot,-,\ast)$ be a dual weak left $\star$-brace, $x,y\in S$ and $e\in P(S)$.  Then $(S, \cdot, \ast)$ is completely regular, orthodox and locally inverse and $xe=e+x$.
\end{lemma}
\begin{proof}We first observe that $xy=x(y-y)y=xy^\ast y y$ by Lemma \ref{liwei}. So Lemma \ref{zicu9} gives that $(S, \cdot, \ast)$ is completely regular, orthodox and locally inverse.
On the other hand, by Lemmas \ref{lw6} and \ref{liwei}, we have
$$xe=x(e+e)=xe-x+xe=xe-xe+x=(xe)^\ast xe +x=ex^\ast x e+x=e+x^\ast x +e+x$$$$=e+x-x+e+x=e+x-x+e+x-x+x=e+x-x+x=e+x.$$
This gives that $xe=e+x$.
\end{proof}
\begin{defn}
 A  dual weak left $\star$-brace $(S,+,\cdot,-,\ast)$ is called a {\em square skew left brace }if $(S, \cdot, \ast)$ is a completely simple semigroup.
\end{defn}
We first give the structure of a square skew left brace.
\begin{lemma}\label{square1}Let $I$ be a non-empty set and $(G,+,\cdot)$ is a skew left brace. Define two multiplications $\cdot$ and $+$ and two unary operations $\ast$ and $-$ on $S=I\times G\times I$ as follows: For all $(i,g,j), (k,h,l)\in S$,
$$(i,g,j)(k,h,l)=(i,gh,l),\, (i,g,j)^\ast=(j,g^{-1}, i);$$$$(i,g,j)+(k,h,l)=(k,g+h,j), -(i,g,j)=(j, -g,i).$$
Then $(S,+,\cdot,-,\ast)$ is a square skew left brace. Conversely, every  square skew left brace can be obtained in this way.
\end{lemma}
\begin{proof}
By simple calculations, we can check the direct part routinely. Now we consider the converse part. Let $(S,+,\cdot,-,\ast)$  be a square skew left brace. Denote $J=P(S)$ and fix an element $e$ in $P(S)$. Denote $$H=\{h\in S\mid hh^\ast=h^\ast h=e\}.$$ Then it is routine to check that $(H, +, \cdot)$ forms a skew left brace with identity $e$ and $h^{-1}=h^\ast$ for all $h\in H$. By the direct part, we obtain a square skew left brace $T=J\times H\times J$. Define a map
$\psi: T\rightarrow S, (i, g, j)\mapsto igj$. Then $ge=g$ and $ejke=e, i^\ast=i,j^\ast=j$ by Lemmas \ref{1.2} and \ref{wanquandan}, and so
$$\psi((i,g,j)(k,h,l))=\psi(j,gh,l)=jghl=igehl$$$$=igejkehl=igjkhl=\psi((i,g,j)) \psi((k,h,l)),$$
$$\psi((i,g,j)^\ast)=\psi(j,g^\ast, i)=jg^\ast i=j^\ast g^\ast i^\ast=(igj)^\ast=(\psi((i,g,j)))^\ast.$$ This gives that $\psi$   preserves the unary operations $\cdot$ and $\ast$.

Let $a\in S$. Then $aa^\ast, a^\ast a\in P(S)=J$. Observe that $$(eae)(eae)^\ast=eaee^\ast a^\ast e^\ast=eaea^\ast e=e$$ by Lemma \ref{wanquandan}. Dually, $(eae)^\ast(eae)=e$. This implies  $eae\in H$ and $(aa^\ast, eae, a^\ast a)\in T$. Moreover, $$\psi(aa^\ast, eae, a^\ast a)=aa^\ast eae a^\ast a=aa^\ast eaa^\ast aea^\ast a=aa^\ast a=a$$ by Lemmas \ref{1.2} and \ref{wanquandan}. This shows that $\psi$ is surjective.

Let $(i,g,j), (k, h, l)\in T$ and $$igj=\psi((i,g,j))=\psi((k,h,l))=khl.$$ Then we  have $iegej=kehel$ as $e$ is the identity of $H$, and so
$$g=ege=eiegeje=ekehele=ehe=h$$ by Lemma \ref{wanquandan}. Thus $igj=kgl$. Observe that $ge=e=gg^{-1}, eje=e$ and $$igjeg^{-1}=igejeg^{-1}=igeg^{-1}=igg^{-1}=ie$$ by Lemma \ref{wanquandan}. Dually,  $kgl eg^{-1}=ke$. This implies that $ie=ke$, and so $i=iei=kei=ki$ by Lemma \ref{wanquandan}. Dually, we have $k=ik$. So $i=i^\ast=(ki)^\ast=i^\ast k^\ast =ik=k$. By symmetry, we obtain that $j=l$. Thus $(i,g,j)=(k, h, l)$. This give that  $\psi$ is injective.

On the other hand, let $(i,g,j), (k,h,l)\in S$. Then   $elie=e$ by Lemma \ref{wanquandan}.
In view of Lemmas \ref{lw6}, \ref{liwei} and \ref{likang5} and the fact that $i,j,k,l,e\in P(S), g,h\in H$ and $e$ is the identity of the skew brace $H$, we obtain
$$\psi((i,g,j)+(k,h,l))=\psi(k,g+h, j)=k(g+h)j=(g+h+k)j=j+g+h+k$$$$=j+g+e+h+k=j+g+elie+h+k=j+g+e+i+l+e+h+k$$$$=j+g+i+l+h+k=igj+khl=\psi((i,g,j))+\psi((k,h,l)),$$
$$\psi(-(i,g,j))=\psi((j,-g,i))=j(-g)i$$$$=i-g+j=-i-g-j=-(j+g+i)=-(igj)=-\psi((i,g,j)).$$
By the above statements, we have $T\cong S$ as dual weak left $\star$-braces.
\end{proof}
We next give a construction method for general dual weak left $\star$-braces. Recall that a {\em lower semilattice} $(Y, \leq)$ is a partially ordered set  in which any two elements $\alpha$ and $\beta$ have a greatest lower bound $\alpha \beta$.
\begin{theorem}\label{main1}
Let $Y$ be a lower semilattice and $\left\{(B_{\alpha}, +, \cdot, -, \ast)\ \left|\ \alpha \in Y\right.\right\}$ be a family of disjoint square skew left braces. For each pair $\alpha,\beta$ of elements of $Y$ such that $\alpha \geq \beta$, let $\phii{\alpha}{\beta}: B_{\alpha}\to B_{\beta}$ be a  square skew left brace homomorphism such that
\begin{enumerate}
    \item[(1)] $\phii{\alpha}{\alpha}$ is the identical automorphism of $B_{\alpha}$, for every $\alpha \in Y$,
    \item[(2)] $\phii{\beta}{\gamma}{}\phii{\alpha}{\beta}{} = \phii{\alpha}{\gamma}{}$, for all $\alpha, \beta, \gamma \in Y$ such that $\alpha \geq \beta \geq \gamma$.
\end{enumerate}
Define two binary operations $+$ and $\cdot$ on $S = \bigcup\left\{B_{\alpha}\ \left|\ \alpha\in Y\right.\right\}$ as follow:
\begin{align*}
    a+b= \phii{\alpha}{\alpha\beta}(a)+\phii{\beta}{\alpha\beta}(b)
    \quad\text{and}\quad
     a b= \phii{\alpha}{\alpha\beta}(a) \phii{\beta}{\alpha\beta}(b),
\end{align*}
for all $a\in B_{\alpha}$ and $b\in B_{\beta}$. Then $(S, +, \cdot, -, \ast)$ forms a dual weak left $\ast$-brace,  called a \emph{strong semilattice of square skew left braces $B_{\alpha}$} and denoted by $S=[Y; B_\alpha;\phii{\alpha}{\beta}]$.  Conversely, any dual weak left $\star$-brace can be constructed in this way.
\end{theorem}
\begin{proof}The direct part can be checked by routine calculations. To show the converse part, let $(S, +, \cdot, -, \ast)$ be a dual weak left $\star$-brace.
Then $(S, +, -)$ and $(S, \cdot, \ast)$ are completely regular, orthodox and locally inverse by Lemmas \ref{zicu9tiaoxingzhi} and \ref{likang5}. Since $xx^\ast=-x+x$ and $x^\ast x=x-x$ for all $x\in S$, we have ${\mathcal L}^+={\mathcal R}^\cdot$ and ${\mathcal R}^+={\mathcal L}^\cdot$ by Lemma \ref{1.2}, and so $ {\mathcal D}^+={\mathcal L}^+\circ {\mathcal R}^+={\mathcal R}^\cdot\circ {\mathcal L}^\cdot={\mathcal D}^\cdot.$
Denote ${\mathcal D}={\mathcal D}^+={\mathcal D}^\cdot.$ By \cite[Theorem II.4.5]{Petrich-Reilly},   ${\mathcal D}={\mathcal J}^+={\mathcal J}^\cdot$.
Denote $Y=S/{\mathcal D}=\{a{\mathcal D} \mid a\in S\}$.
By the statements before \cite[Theorem 4.1.3]{Howie}, $\mathcal D$ is a congruence on both $(S, +)$ and $(S, \cdot)$ such that $(Y,+)$ and $(Y, \cdot)$ are semilattices, respectively. By Lallement's Lemma (\cite[Lemma 2.4.3]{Howie}), we obtain that $Y=\{e{\mathcal D}\mid e\in E(S)\}$. By Lemma \ref{liwei},
$$e{\mathcal D}+f{\mathcal D}=(e+f){\mathcal D}=(fe){\mathcal D}=(f{\mathcal D}) (e{\mathcal D})=(e{\mathcal D}) (f{\mathcal D}) $$for all $e,f\in E(S)$.
This gives that $(Y, +)=(Y, \cdot)$. For all $\alpha\in Y$, denote $S_\alpha=\{a\in S\mid a{\mathcal D}=\alpha\}$ (i.e. each $S_\alpha$ is exact a $\mathcal D$-class of $S$). By the statements before \cite[Theorem 4.1.3]{Howie},  both $(S_\alpha, +)$ and $(S_\alpha, \cdot)$ are completely simple semigroups, and
\begin{equation}\label{xxx1}
S_\alpha S_\beta\subseteq S_{\alpha\beta}\mbox{ and }S_\alpha +S_\beta\subseteq S_{\alpha\beta}  \mbox{ for all } \alpha, \beta\in Y.
\end{equation}
Furthermore, we have $a{\mathcal D} a^\ast$ and $a\mathcal{D}(-a)$ for all $a\in S$ by Lemma \ref{1.2}, and so each $S_\alpha$  is closed under $-$ and $\ast$, respectively. Thus $(S_\alpha, +,\cdot, -,\ast)$ forms a square skew left brace.

Assume that $\alpha, \beta\in Y$ with $\alpha\geq \beta$ and fix an element $e\in P(S_\beta)$. By (\ref{xxx1}), we define
\begin{equation}\label{likang2}
\psi_{\alpha, \beta}: S_\alpha\rightarrow S_\beta, a \mapsto aa^\ast e a.
\end{equation}
Let $a\in S_\alpha$ and $f\in P(S_\beta)$. Then $$e+f+e=efe=e \mbox{ and }f+e+f=fef=f$$ by Lemma \ref{wanquandan} and \ref{liwei}. This together with Lemma \ref{zicu8} gives that $$aa^\ast ea=aa^\ast efe a=aa^\ast eefa=aa^\ast efa=aa^\ast eff a=aa^\ast fef a=aa^\ast f a.$$This yields that the definition of $\psi_{\alpha, \beta}$ is independent to the choice of $e$ in $P(S_\beta)$. In particular, since $aea^\ast\in P(S_\beta)$ by Lemma \ref{1.2}, it follows that
\begin{equation}\label{likang1}
\begin{aligned}
&a+e-a+a= a+e+aa^\ast =aa^\ast ea=aa^\ast (aea^\ast) a\\
&=aea^\ast a =(e+a)a^\ast a=a^\ast a+e+a=a-a+e+a.
\end{aligned}
\end{equation}
by (\ref{lw5}), Lemmas \ref{liwei} and \ref{likang5} and Definition \ref{xxx2}.
Let $\alpha, \beta, \gamma\in Y$ with $\alpha\geq \beta \geq \gamma$ and fix $e\in P(S_\alpha), f\in P(S_\beta), g\in P(S_\gamma)$.
Then  $gfg\in  P(S_\gamma)$ by (\ref{xxx1}). Since $S_\gamma$ is a square skew left brace, we have  $gfg=g(gfg)g=g$ by Lemma \ref{wanquandan}.
Let $a,b\in S_\alpha$. Then $a^\ast fa\in P(S)$ by Lemma \ref{1.2}. In view of Lemma \ref{zicu8}, we obtain  $$\psi_{\alpha \beta}(ab)=ab(ab)^\ast fab=abb^\ast a^\ast fab=a a^\ast fa bb^\ast b=aa^\ast fab.$$ By Lemma \ref{1.2} and (\ref{xxx1}), we have $a^\ast fa, a^\ast fa   f   a^\ast fa\in P(S_{\alpha\beta})$. This together with Lemma \ref{wanquandan}
gives that $a^\ast fa   f   a^\ast fa=a^\ast fa  (a^\ast fa   f   a^\ast fa)  a^\ast fa=a^\ast fa$, and so
$$\psi_{\alpha \beta}(a)\psi_{\alpha, \beta}(b)=(aa^\ast fa) (bb^\ast fb)=aa^\ast fa f bb^\ast b=aa^\ast fa fb=aa^\ast fa f a^\ast fa b=aa^\ast fab$$  by Lemma \ref{zicu8}. This shows that $\psi_{\alpha \beta}(ab)=\psi_{\alpha \beta}(a)\psi_{\alpha, \beta}(b)$. By Lemma \ref{1.2} and (\ref{likang1}),
$$\psi_{\alpha, \beta}(a^\ast)=a^\ast a^{\ast\ast}fa^\ast=a^\ast afa^\ast=a^\ast faa^\ast=(aa^\ast fa)^\ast=(\psi_{\alpha, \beta}(a))^\ast.$$
This shows that  $\psi_{\alpha,\beta}$ preserves  $\cdot$ and $\ast$.
By (\ref{likang1}) and the above statements, it follows that $\psi_{\alpha,\beta}$ also preserves  $+$ and $-$. This shows that $\psi_{\alpha,\beta}$  is a square skew left brace homomorphism  for all $\alpha, \beta\in Y$ with $\alpha\leq \beta$. Moreover,  for $a\in S_\alpha$ and $e\in P(S_\alpha)$, we have $aa^\ast e aa^\ast =aa^\ast$ by Lemma \ref{wanquandan}, and  so $\psi_{\alpha,\alpha}(a)=aa^\ast ea=aa^\ast eaa^\ast a=aa^\ast a=a$. This shows that
$\psi_{\alpha, \alpha}$ is the identical automorphism of $S_{\alpha}$ for every $\alpha \in Y$. Finally, since  $f\in P(S_\beta), g\in P(S_\gamma)$ and $\beta\geq \gamma$, we have $g, gfg\in P(S_\beta)$ by Lemma \ref{1.2} and (\ref{xxx1}), and so  $gfg= g(gfg)g=g$ by Lemma \ref{wanquandan}. This together with Lemma \ref{zicu8} gives that
$$\psi_{\beta, \gamma}\psi_{\alpha, \beta}(a)=\psi_{\beta, \gamma}(aa^\ast fa)=aa^\ast fa  (aa^\ast fa)^\ast g  aa^\ast fa
$$$$=aa^\ast fa  a^\ast faa^\ast   g  aa^\ast fa=aa^\ast gf a=aa^\ast gfg a=aa^\ast ga=\psi_{\alpha,\gamma}(a).$$
So $\psi_{\beta, \gamma}\psi_{\alpha, \beta}=\psi_{\alpha,\gamma}$. By direct part, we have the  dual weak left $\ast$-brace $[Y; S_\alpha;\psi_{\alpha,\beta}]$.
Let $a\in S_\alpha$ and $b\in S_{\beta}$ and fix $h\in P(S_{\alpha\beta})$. Then by Lemmas \ref{1.2}, \ref{zicu8} and \ref{wanquandan},  we have
$$\psi_{\alpha, \alpha\beta}(a)\psi_{\alpha, \alpha\beta}(b)=aa^\ast ha (bb^\ast h)b=aa^\ast ha (h bb^\ast) b=aa^\ast ha hb=a[(a^\ast ha) h] b=a[h a^\ast h a] hb$$$$=a[ h (a^\ast ha) h ]b\overset{{\rm Lemma} \ref{wanquandan}}{=}ahb\overset{{\rm Lemma} \ref{zicu8}}{
=}a[a^\ast a bb^\ast h  aa^\ast bb^\ast] b\overset{{\rm Lemma} \ref{wanquandan}}{=}aa^\ast a bb^\ast   b=ab.$$
By (\ref{likang1}) and the above statements,  we can also prove that $\psi_{\alpha, \alpha\beta}(a)+\psi_{\alpha, \alpha\beta}(b)=a+b$ for all $a\in S_\alpha$ and $b\in S_{\beta}$.
Thus $(S, +, \cdot, -, \ast)$ is indeed $[Y; S_\alpha;\psi_{\alpha,\beta}]$.
\end{proof}
\section{Deformed solutions on dual weak left $\star$-braces}
In this section, we introduce  distributors for dual weak left $\star$-braces and give some structural properties of them, and prove that the map deformed by each  distributor  is always a solution of the Yang-Baxter equation.

Let $(S,+,\cdot,-,\ast)$ be a dual weak left $\star$-brace. Then we call
\begin{align*}
    \mathcal{D}_r(S)=\{z \in S  \mid (a+b)z=az-z+bz \mbox{ for all }a,b \in S\}
\end{align*}
the set of \emph{right distributors} of $(S,+,\cdot,-,\ast)$.
\begin{lemma}\label{shou13}Let $S=I\times G\times I$  be a square skew left brace constructed in Lemma \ref{square1} and $z=(u,t,v)\in S$.  Then
$z\in \mathcal{D}_r(S)$ if and only if $t\in \mathcal{D}_r(G)$.
\end{lemma}
\begin{proof}Let $a=(i,x,j), b=(k,y,l)\in S$. Then
$$(a+b)z=(k,x+y,j)(u,t,v)=(k, (x+y)t, v),$$
$$az-z+bz=(k,xz-t+yt,v)$$
This implies that $(a+b)z=az-z+bz$ for all $a,b\in S$ if and only if $xz-t+yt$ for all $x,y\in G$. Thus the desired result holds.
\end{proof}
\begin{lemma}\label{shou15}Assume that  $S=[Y; S_\alpha;\phii{\alpha}{\beta}]$ is a strong semilattice of a family of square skew left braces and $z\in S_\delta$ for some $\delta\in Y$. If $z\in  \mathcal{D}_r(S)$, $\xi\in  Y$ and $\xi\leq \delta$,  then $\phi_{\delta, \xi}(z)\in \mathcal{D}_r(S_\xi)$.
\end{lemma}
\begin{proof}
Let $p,q\in S_\xi$. Then we have $$(p+q)\widetilde{z}=(p+q)z=pz-z+qz=p\widetilde{z}-\widetilde{z}+q\widetilde{z}$$ as $z\in  \mathcal{D}_r(S)$ and $\xi\leq \delta$. This shows that $\widetilde{z}\in  \mathcal{D}_r(S_\xi)$.
\end{proof}

\begin{lemma}\label{shou14} Let $(S,+,\cdot,-,\ast)$ be a dual weak left $\star$-brace and $z\in S$. Then  $z\in \mathcal{D}_r(S)$ if and only if $(a-b+c)z=az-bz+cz$ for all $a,b,c\in S$.
\end{lemma}
\begin{proof}By Theorem \ref{main1}, we assume that  $S=[Y; S_\alpha;\phii{\alpha}{\beta}]$ is a strong semilattice of a family of square skew left braces and $z\in S_\delta$ for some $\delta\in Y$.

Assume that the given condition holds. Let $\alpha,  \gamma\in Y$ and $a\in S_\alpha$ and $c\in S_\gamma$.
By Lemma \ref{square1}, we can let $S_{\alpha\gamma\delta}=I\times G\times I$, where $(G, +, \cdot)$ is a skew left brace with identity $1$ and $I$ is non-empty set. Denote
$$\overline{a}=\phii{\alpha}{\alpha\gamma\delta}(a)=(i,x,j), \overline{c}=\phii{\gamma}{\alpha\gamma\delta}(c)=(m,w,n), \overline{z}=\phii{\delta}{\alpha\gamma\delta}(z)=(u,t,v).$$  By the given axiom, we have $$az-z+cz=az-zz^\ast z+cz=(a-zz^\ast+c)z=(\overline{a}-\overline{z}~\overline{z}^\ast+\overline{c})\overline{z}$$$$=(m, (x+w)t,v)=(\overline{a}+\overline{c})\overline{z}=(a+c)z.$$
Thus $z\in \mathcal{D}_r(S)$.

Conversely, assume that  $z\in \mathcal{D}_r(S)$.  Let $\alpha, \beta, \gamma\in Y$ and $a\in S_\alpha, b\in S_\beta, c\in S_\gamma$.
By Lemma \ref{square1} again, we  let $\xi=\alpha\beta\gamma\delta$ and $S_{\xi}=I\times G\times I$, where $(G, +, \cdot)$ is a skew left brace with identity $1$ and $I$ is non-empty set. Denote $\widetilde{z}=\phii{\delta}{\xi}(z)=(u,t,v).$  By Lemma \ref{shou15}, $\widetilde{z}\in \mathcal{D}_r(S_\xi)$.  By Lemma \ref{shou13}, $t\in \mathcal{D}_r(G)$. Let
$$\widetilde{a}=\phii{\alpha}{\xi}(a)=(i,x,j), \widetilde{b}=\phii{\beta}{\xi}(b)=(k,y,l), \widetilde{c}=\phii{\gamma}{\xi}(c)=(m,w,n).$$
By Lemma \ref{ma}, we have $(x-y+w)t=xt-yt+wt$ and so
$$(a-b+c)z=(\widetilde{a}-\widetilde{b}+\widetilde{c})\widetilde{z}=(m, (x-y+w)t, v)$$$$=(m, xt-yt+wt, v)=\widetilde{a}~\widetilde{z}-\widetilde{b}\widetilde{z}+\widetilde{c}~\widetilde{z}=az-bz+cz.$$
This gives the condition stated in the lemma.
\end{proof}

\begin{prop}Let $(S,+,\cdot,-,\ast)$ be a dual weak left $\star$-brace. Then $\mathcal{D}_r(S)$ is a regular $\star$-subsemigroup of $(S, \cdot, \ast)$ such that $P(S)\subseteq \mathcal{D}_r(S)$.
\end{prop}
\begin{proof}Let $e\in P(S)$ and $a,b\in S$. Then by Lemmas \ref{lw6}, \ref{liwei} and \ref{likang5}, we have $$(a+b)e=e+a+b=e+a+e+b=e+a-e+e+b=ae-e+be,$$ which shows $e\in \mathcal{D}_r(S)$. So  $P(S)\subseteq \mathcal{D}_r(S)$.
Now let $a,b\in S$ and $z_1, z_2\in \mathcal{D}_r(S)$. Then $$(a+b)(z_1z_2)=[(a+b)z_1]z_2=(az_1-z_1+bz_1)z_2=az_1z_2-z_1z_2+bz_1z_2$$ by Lemma \ref{shou14}. This implies that $z_1 z_2\in \mathcal{D}_r(S)$.

Let $a,b\in S$ and $z\in \mathcal{D}_r(S)$. By Theorem \ref{main1}, we assume that  $S=[Y; S_\alpha;\phii{\alpha}{\beta}]$ is a strong semilattice of a family of square skew left braces and $z\in S_\delta$ for some $\delta\in Y$. Let $\alpha, \beta\in Y$ and $a\in S_\alpha, b\in S_\beta$. Denote $\xi=\alpha\beta\delta$, $S_\xi=I\times G\times I$ by Lemma \ref{square1} and $$\overline{a}=\phii{\alpha}{\xi}(a)=(i,x,j), \overline{b}=\phii{\gamma}{\xi}(b)=(k,y,l), \overline{z}=\phii{\delta}{\xi}(z)=(u,t,v).$$
By Lemmas \ref{shou13} and \ref{shou15}, we have $t\in \mathcal{D}_r(G)$. By Lemma \ref{zhuyin1} (1), $\mathcal{D}_r(G)$ is a subgroup of $(G,\cdot)$, and so $t^{-1}\in \mathcal{D}_r(G)$. This implies that $(x+y)t^{-1}=xt^{-1}-t^{-1}+yt^{-1}$, and hence
$$(a+b)z^\ast=(\overline{a}+\overline{b})\overline{z}^\ast=(k,(x+y)t^{-1},v)$$$$=(k,xt^{-1}-t^{-1}+yt^{-1},v)=\overline{a}~\overline{z}^\ast -\overline{z}^\ast +\overline{b}\overline{z}^\ast=az^\ast-z^\ast+bz^\ast.$$
This gives that $z^\ast\in \mathcal{D}_r(S)$. Thus  the desired result is true.
\end{proof}

Let $(S,+,\cdot,-,\ast)$ be a dual weak left $\star$-brace and $a, b, z\in S$. Define two maps  $\sigma_a^z, \tau_b^z$ from $B$ to $B$  by setting
\begin{align*}
\sigma_a^{z}(b) =z^\ast z (-a z+a  b z)z^\ast z
\mbox{ and }
\tau_b^{z}(a) =(\sigma_a^{z}(b))^\ast ab=z^\ast z \left(-a  z+a  b  z\right)^{\ast} z^\ast z  a   b,
\end{align*}
Moreover, define a map ${r}^S_z$ from $S\times S$ to $S\times S$ by setting
$${r}^S_z(a,b)=\left({\sigma}^z_a(b), {\tau}^z_b(a)\right)=(z^\ast z (-a z+a  b z)z^\ast z, z^\ast z \left(-a  z+a  b  z\right)^{\ast} z^\ast z  a   b)$$
for all $a, b\in S$.

\begin{lemma}\label{shou8}Let $(S,+,\cdot,-,\ast)$  be a square skew left brace and $z\in S$. Then ${r}^S_z$ is a solution of the Yang-Baxter equation if and only if $z\in  \mathcal{D}_r(S)$.
\end{lemma}
\begin{proof}By Lemma \ref{square1}, we assume that $S=I\times G\times I$, where $(G, +, \cdot)$ is a  skew left brace with identity $1$, and let $$a=(i,x,j), b=(k,y,l), c=(m,w,n), z=(u,t,v)\in S.$$ Then $z^\ast z=(v,1,v)$ and
$${\sigma}^z_a(b)=z^\ast z (-a z+a  b z)z^\ast z=(v,-xt+xyt,v)=(v,\hat{\sigma}^t_x(y), v),$$
$$\tau_b^{z}(a) =(\sigma_a^{z}(b))^\ast ab=(v, (\hat{\sigma}^t_x(y))^{-1} xy, l)=(v, \hat{\tau}_y^{t}(x),l).$$
This implies that
\begin{equation}\label{shou2}
\begin{aligned}
&\sigma^{z}_a\sigma^{z}_b(c)=\sigma^{z}_a(v,\hat{\sigma}^t_y(w), v)=(v,\hat{\sigma}^t_x\hat{\sigma}^t_y(w),v),\\
&  \sigma^{z}_{\sigma^{z}_a\left(b\right)}\sigma^{z}_{\tau^{z}_b\left(a\right)}\left(c\right)
=(v,\hat{\sigma}^{t}_{\hat{\sigma}^{t}_x\left(y\right)}\hat{\sigma}^{t}_{\hat{\tau}^{t}_y\left(x\right)}\left(w\right),v),
\end{aligned}
\end{equation}

$$\tau^{z}_c\tau^{z}_b\left(a\right)=\tau^{z}_c(v, \hat{\tau}^{t}_y(x), l)=(v, \hat{\tau}^{t}_w\hat{\tau}^{t}_y(x), n),\,\,\,\,
\tau^{z}_{\tau^{z}_c\left(b\right)}\tau^{z}_{\sigma^{z}_b\left(c\right)}\left(a\right)
=(v,\hat{\tau}^{t}_{\hat{\tau}^{t}_w\left(y\right)}\hat{\tau}^{w}_{\hat{\sigma}^{t}_y\left(w\right)}\left(x\right),n)$$
$$\sigma^{z}_{\tau^{z}_{\sigma^{z}_b\left(c\right)}\left(a\right)}\tau^{z}_c\left(b\right)
=(v,\hat{\sigma}^{t}_{\hat{\tau}^{t}_{\hat{\sigma}^{t}_y\left(w\right)}\left(x\right)}\hat{\tau}^{t}_w\left(y\right),v),\,\,\,\, \tau^{z}_{\sigma^{z}_{\tau^{z}_b\left(a\right)}\left(c\right)}\sigma^{z}_a\left(b\right)
=(v,\hat{\tau}^{t}_{\hat{\sigma}^{t}_{\hat{\tau}^{t}_y\left(x\right)}\left(w\right)}\hat{\sigma}^{t}_x\left(y\right),v).$$
By Lemmas \ref{zhuyin1}, \ref{shou13} and (\ref{jie1})-(\ref{jie3}), we obtain that
$$
r^S_z \mbox{ is a solution} \Longleftrightarrow r^G_t \mbox{ is a solution} \Longleftrightarrow  t\in  \mathcal{D}_r(G) \Longleftrightarrow  z\in  \mathcal{D}_r(S).
$$
Thus the desired result follows.
\end{proof}

\begin{theorem}\label{main2}Let $(S,+,\cdot,-,\ast)$  be a dual weak left $\star$-brace and $z\in S$. Then ${r}^S_z$ is a solution of the Yang-Baxter equation if and only if $z\in  \mathcal{D}_r(S)$.
\end{theorem}
\begin{proof}
By Theorem \ref{main1}, we assume that  $S=[Y; S_\alpha;\phii{\alpha}{\beta}]$ is a strong semilattice of a family of square skew left braces and $z\in S_\delta$ for some $\delta\in Y$.

Assume that ${r}^S_z$ is a solution of the Yang-Baxter equation. By (Y1),
\begin{equation}\label{shou3}
\sigma^{z}_a\sigma^{z}_b(c)=\sigma^{z}_{\sigma^{z}_a\left(b\right)}\sigma^{z}_{\tau^{z}_b\left(a\right)}\left(c\right) \mbox{ for all } a,b\in S.
\end{equation}
Let $\alpha,\beta\in Y$ and $a\in S_\alpha, b\in S_\beta$. By Lemma \ref{square1}, we can let $S_{\alpha\beta\delta}=I\times G\times I$, where $(G, +, \cdot)$ is a skew left brace with identity $1$ and $I$ is non-empty set. Denote
$$\overline{a}=\phii{\alpha}{\alpha\beta\delta}(a)=(i,x,j), \overline{b}=\phii{\beta}{\alpha\beta\delta}(b)=(k,y,l), \overline{z}=\phii{\delta}{\alpha\beta\delta}(z)=(u,t,v).$$
Set $c=a^\ast (az+b)z^\ast$ and $\overline{c}=\overline{a}^\ast (\overline{a}~\overline{z}+\overline{b})\overline{z}^\ast$. Then we obtain
$\overline{c}=(j, x^{-1}(xt+y)t^{-1}, u)$ and $\overline{a}^\ast\overline{a}=(j,1,j)$. In view of (\ref{shou2}) and Lemma \ref{zhuyin1} (4), we have
\begin{equation}\label{shou4} \sigma^{\overline{z}}_{\overline{a}}\sigma^{\overline{z}}_{\overline{a}^\ast\overline{a}}(\overline{c})=(v,\hat{\sigma}^t_x\hat{\sigma}^t_1(x^{-1}(xt+y)t^{-1}),v)
=(v,-xt+(x+y)t,v),
\end{equation}
\begin{equation}\label{shou5}
\begin{aligned}
&\sigma^{\overline{z}}_{\sigma^{\overline{z}}_{\overline{a}}\left(\overline{a}^\ast\overline{a}\right)}
\sigma^{\overline{z}}_{\tau^{\overline{z}}_{\overline{a}^\ast\overline{a}}\left(\overline{a}\right)}\left(\overline{c}\right)
=(v,\hat{\sigma}^{t}_{\hat{\sigma}^{t}_x\left(1\right)}\hat{\sigma}^{t}_{\hat{\tau}^{t}_1\left(x\right)}\left(\overline{c}\right),v)\\
&=(v,\hat{\sigma}^{t}_{1}\hat{\sigma}^{t}_{x}\left(x^{-1}(xt+y)t^{-1}\right),v)=(v,-t+yt,v).
\end{aligned}
\end{equation}
Observe that
\begin{equation}\label{shou6}
\sigma^z_a(b)=z^\ast z(-az+abz)z^\ast z=\overline{z}^\ast \overline{z}(-\overline{a}~\overline{z}+\overline{a}~\overline{b}~\overline{z})\overline{z}^\ast \overline{z}=\sigma^{\overline{z}}_{\overline{a}}(\overline{b}),\,\, \tau^z_b(a)=\tau^{\overline{z}}_{\overline{b}}(\overline{a}).
\end{equation}
 This together with (\ref{shou3}), (\ref{shou4}) and (\ref{shou5}) implies  that
$-xt+(x+y)t=-t+yt$ in the skew left brace $G$. So $(x+y)t=xt-t+yt$. Observe that $$(a+b)z=(\overline{a}+\overline{b})\overline{z}=(k, (x+y)t, v)$$ and $$az-z+bz=\overline{a}~\overline{z}-\overline{z}+\overline{b}~\overline{z}=(k, xt-t+yt, v),$$ it follows that $(a+b)z=az-z+bz$. Thus $z\in  \mathcal{D}_r(S)$.

Conversely, let $z\in  \mathcal{D}_r(S)$ and $ z\in S_\delta$. Let $\alpha,\beta, \gamma\in Y$ and $a\in S_\alpha, b\in S_\beta, c\in S_\gamma$.
Denote $\xi=\alpha\beta\gamma\delta$. Then $\xi\leq \delta$. Let $\overline{z}=\phi_{\delta, \xi}(z)$. Then $\overline{z}\in S_\xi$.
By Lemma \ref{shou15}, we have  $\overline{z}\in  \mathcal{D}_r(S_\xi)$.  Denote $\overline{a}=\phi_{\alpha, \xi}(a), \overline{b}=\phi_{\beta, \xi}(b),  \overline{c}=\phi_{\gamma, \xi}(c)$.
Then $\overline{a}, \overline{b}, \overline{c}, \overline{z}\in S_\xi$. Since $\overline{z}\in  \mathcal{D}_r(S_\xi)$ and $S_\xi$ is a square skew left brace, it follows that $r^{S_\xi}_{\overline{z}}$ is a solution of the Yang-Baxter equation by Lemma \ref{shou8}.  By (\ref{jie1}) (for $r^{S_\xi}_{\overline{z}}$) and (\ref{shou6}), we obtain $$\sigma^{z}_a\sigma^{z}_b(c)=\sigma^{z}_{\overline{a}}\sigma^{\overline{z}}_{\overline{b}}(\overline{c})
=\sigma^{\overline{z}}_{\sigma^{\overline{z}}_{\overline{a}}\left(\overline{b}\right)}
\sigma^{\overline{z}}_{\tau^{\overline{z}}_{\overline{b}}\left(\overline{a}\right)}\left(\overline{a}\right)
=\sigma^{z}_{\sigma^{z}_a\left(b\right)}\sigma^{z}_{\tau^{z}_b\left(a\right)}\left(c\right).$$
This shows that (\ref{jie1}) holds for $r^S_{z}$.  Similarly, we can show that both (\ref{jie2}) and (\ref{jie3}) are true for $r^S_{z}$. Thus $r^S_{z}$ is a solution of the Yang-Baxter equation.
\end{proof}

\begin{theorem}\label{theor_rropr}
Let $(S,+, \cdot, -, \ast)$ be a dual weak left $\star$-brace and $z \in \mathcal{D}_r(S)$. Define  a map $\check{r}^S_{z}$ from $S\times S$ to $S\times S$  by setting
\begin{align*}
     \check{r}^S_{z}(a,b)=\left((ab-az+z)zz^\ast, \ \left((ab-az+z)zz^\ast\right)^\ast ab \right) \mbox{ for all } a,b \in S.
\end{align*}
For brevity, we write $r^S_z$ and $\check{r}^S_{z}$ by $r_z$ and $\check{r}_{z}$  respectively. Then $\check{r}^S_{z}$ is a solution of the Yang-Baxter equation.  Moreover, we have
\begin{align}\label{shou7}
      r_{z}  \check{r}_{z^\ast}  r_{z} = r_{z}, \qquad
      \check{r}_{z^\ast}  r_{z}  \check{r}_{z^\ast}= \check{r}_{z^\ast}, \qquad  r_{z}\check{r}_{z^\ast} = \check{r}_{z^\ast}r_{z}.
\end{align}
\end{theorem}
\begin{proof}
By Lemma \ref{zhuyin1} (2) and using the similar method to the proof of corresponding parts in Lemma \ref{shou8} and Theorem \ref{main2}, we can show that  $\check{r}_{z}$ is indeed a solution of the Yang-Baxter equation. We omit the details.

Now, we  prove (\ref{shou7}).   In view of the method used in the proof of Lemma \ref{shou8} and Theorem \ref{main2},  we can assume that $S=I\times G\times I$, where $(G, +, \cdot)$ is a  skew left brace with identity $1$. Let $a=(i,x,j), b=(k,y,l), z=(u,t,v)\in S$. Then $z^\ast=(v,t^{-1}, u)$ and
\begin{align}\label{shou11}
(ab-az^\ast+z^\ast)z^\ast z=(v, xy-xt^{-1}+t^{-1}, l)(v,1,v)=(v,\check{\sigma}_x^{t^{-1}}(y),v),
\end{align}
\begin{equation}\label{shou12}
\begin{aligned}
&((ab-az^\ast+z^\ast)z^\ast z)^{\ast} ab=(v,(\check{\sigma}_x^{t^{-1}}(y))^{-1},v)(i,xy,l)\\
&=(v, (\check{\sigma}_x^{t^{-1}}(y))^{-1} xy,l)=(v, \check{\tau}_y^{t^{-1}}(x),l)=(v, (xy-xt^{-1}+t^{-1})^{-1}, l).
\end{aligned}
\end{equation}
By Lemma \ref{zhuyin1} (2), $\check{r}^G_{t^{-1}}$ is the inverse of $r^G_t$, this implies that
$$z^\ast z(-(v,\check{\sigma}_x^{t^{-1}}(y),v)z+(v,\check{\sigma}_x^{t^{-1}}(y),v)(v, \check{\tau}_y^{t^{-1}}(x),l)z)z^\ast z$$
$$=(v,-\check{\sigma}_x^{t^{-1}}(y)~t+\check{\sigma}_x^{t^{-1}}(y)~\check{\tau}_y^{t^{-1}}(x)~t ,v)=(v,\hat{\sigma}^t_{\check{\sigma}_x^{t^{-1}}(y)}(\check{\tau}_y^{t^{-1}}(x)),v)=(v,x,v),$$
$$(v,x,v)^\ast (v,\check{\sigma}_x^{t^{-1}}(y),v), (v, \check{\tau}_y^{t^{-1}}(x),l)=(v,x^{-1},v)(v, xy, l)=(v, y, l).$$
Therefore
$$ r_{z} \check{r}_{z^\ast}(a, b)=r_{z}((v,\check{\sigma}_x^{t^{-1}}(y),v), (v, \check{\tau}_y^{t^{-1}}(x),l))=((v,x,v),(v,y,l)).$$
On the other hand, we have
\begin{equation}\label{shou9}
z^\ast z(-az+abz)z^\ast z=(v, -xt+xyt,v)=(v,\hat{\sigma}^t_x(y),v),
\end{equation}
\begin{equation}\label{shou10}
\begin{aligned}
&(z^\ast z(-az+abz)z^\ast z)^{\ast}ab=(v,\hat{\sigma}^t_x(y),v)^\ast (i, xy, l)\\
&=(v, (\hat{\sigma}^t_x(y))^{-1} xy, l)=(v,\hat{\tau}^t_y(x), l)=(v, (-xt+xyt)^{-1}xy, l).
\end{aligned}
\end{equation}
By Lemma \ref{zhuyin1} (2), $\check{r}^G_{t^{-1}}$ is the inverse of $r^G_t$, this implies that
$$[(v,\hat{\sigma}^t_x(y),v) (v,\hat{\tau}^t_y(x), l)-(v,\hat{\sigma}^t_x(y),v)z^\ast +z^\ast]z^\ast z$$
$$=(v, \hat{\sigma}^t_x(y) \hat{\tau}^t_y(x)-\hat{\sigma}^t_x(y) t^{-1}+t^{-1}, v)=(v,\check{\sigma}^{t^{-1}}_{\hat{\sigma}^t_x(y)}(\hat{\tau}^t_y(x)),v)=(v,x,v),$$
$$(v,x,v)^\ast (v,\hat{\sigma}^t_x(y),v) (v,\hat{\tau}^t_y(x), l)=(v,x^{-1},v)(v, xy, l)=(v, y, l).$$
Therefore
$$\check{r}_{z^\ast} r_{z}(a, b)=\check{r}_{z^\ast}((v,\hat{\sigma}_x^{t}(y),v), (v, \hat{\tau}_y^{t}(x),l))=((v,x,v),(v,y,l)).$$
Thus $r_{z}\check{r}_{z^\ast} = \check{r}_{z^\ast}r_{z}$.
Observe that
$$z^\ast z(-(u,x,v)z+(v,x,v)(v,y,l)z)z^\ast z=(v, -xt+xyt, v)$$ and $$(v, -xt+xyt, v)^\ast (v,x,v)(v,y,l)=(v,(-xt+xyt)^{-1}xy,l),$$
it follows that $$r_z\check{r}_{z^\ast}r_{z}(a,b)=r_z((v,x,v),(v,y,l))$$$$=((v, -xt+xyt, v), (v,(-xt+xyt)^{-1}xy,l))=r_z(x,y)$$
by (\ref{shou9}) and (\ref{shou10}).  This gives that $r_{z}  \check{r}_{z^\ast}  r_{z} = r_{z}$. Observe that
$$((v,x,v)(v,y,l)-(v,x,v)z^\ast+z^\ast)z^\ast z=(v, xy-xt^{-1}+t^{-1}, v)$$ and $$(v, xy-xt^{-1}+t^{-1}, v)^\ast (v,x,v)(v,y,l)=(v,(xy-xt^{-1}+t^{-1})^{-1}xy,l),$$
it follows that
$$\check{r}_{z^\ast}r_{z}\check{r}_{z^\ast}(a,b)=r_z((v,x,v),(v,y,l))$$$$=((v, xy-xt^{-1}+t^{-1}, v), (v,(xy-xt^{-1}+t^{-1})^{-1}xy,l))=\check{r}_{z^\ast}(a,b)$$ by (\ref{shou11}) and (\ref{shou12}). This gives that $\check{r}_{z^\ast}  r_{z}  \check{r}_{z^\ast}= \check{r}_{z^\ast}$.
\end{proof}
\begin{remark}
Let $(S,+, \cdot, -, \ast)$ be a dual weak left $\star$-brace and $z \in \mathcal{D}_r(S)$. Define two maps $\overline{r}_z$ and $\overline{\check{r}}_{z}$ from $S\times S$ to $S\times S$ by setting
$$\overline{r}_z(a,b)=(-a z+a  b z, \left(-a  z+a  b  z\right)^{\ast}  a   b) \mbox{ for all } a,b \in S,$$
$$\overline{\check{r}}_{z}(a,b)=\left(ab-az+z, \ \left(ab-az+z\right)^\ast ab \right) \mbox{ for all } a,b \in S.$$
Then one can prove that $\overline{r}_z$ and $\overline{\check{r}}_{z}$ are still solutions of Yang-Baxter equation, but  $\overline{r}_{z}  \overline{\check{r}}_{z^\ast}  \overline{r}_{z} \not= \overline{r}_{z}$ and $\overline{r}_{z}\overline{\check{r}}_{z^\ast} \not= \overline{\check{r}}_{z\ast}\overline{r}_{z}$ in general. We also observe that $\overline{r}_{z}   {\check{r}}_{z^\ast}  \overline{r}_{z} \not= \overline{r}_{z}$,  ${r}_{z}\overline{\check{r}}_{z^\ast} \not= {\check{r}}_{z\ast}\overline{r}_{z}$ and $\overline{r}_{z} {\check{r}}_{z^\ast} =  {\check{r}}_{z^\ast}\overline{r}_{z}$ in general.
\end{remark}
 By using Lemma \ref{zhuyin1} (3) and the method used in the proof of Lemma \ref{shou8}, Theorems \ref{main2}, \ref{theor_rropr} and \cite[Proposition 3.11]{Mazzotta-Rybolowicz-Stefanelli}, we can prove the following result, we omit the details.
\begin{prop}
Let $(S,+,\cdot,-,\ast)$  be a dual weak left $\star$-brace and  $z\in  \mathcal{D}_r(S)$. Then
\begin{itemize}
\item[(1)]  For all $a\in S$, we have
$$\sigma_a^{z}\sigma_{a^\ast}^{z^\ast}\sigma_a^{z}
    = \sigma_a^{z},~
     \sigma_{a^\ast}^{z^\ast}\sigma_a^{z}\sigma_{a^\ast}^{z^\ast}
    = \sigma_{a^\ast}^{z^\ast}, ~
    \tau^{z}_{a^\ast}\tau_{a}^{z}\tau^{z}_{a^\ast}
    = \tau^{z}_{a^\ast}, ~\tau^{z}_{a^\ast}\tau_{a}^{z}\tau^{z}_{a^\ast}
    = \tau^{z}_{a^\ast},\tau^{z}_a \tau_{a^\ast}^{z}
    =\tau_{a^\ast}^{z}\tau_a^z,$$ but
 $\sigma_{a}^{z}\sigma_{a^\ast}^{z^\ast}    \neq\sigma_{a^\ast}^{z^\ast}\sigma_{a}^{z}$ in general.
\item[(2)] ${\tau}^z: (S,\cdot)\rightarrow {\rm{Map}}(S), a\mapsto  {\tau}^z_a$ is an anti-homomorphism.
\item[(3)] ${\sigma}^z: (S,\cdot)\rightarrow {\rm{Map}}(S), a\mapsto  {\sigma}^z_a$ is a homomorphism if and only if $cz=z+c$ for all $c\in G$.
\end{itemize}
\end{prop}

\begin{remark}
In \cite{Catino-Mazzotta-Stefanelli2}, Catino,  Mazzotta and Stefanelli have established the Rota-Baxter operator theory of Clifford semigroups and explored the relationship between this theory and dual weak left braces. On the other hand, Gong and Wang \cite{Gong-Wang} considered the relationship among post Clifford semigroups, dual weak left braces and other algebraic structures.
Thus, the following question is natural: How to establish a theory parallel to the theory in \cite{Catino-Mazzotta-Stefanelli2,Gong-Wang}  for dual weak left $\star$-braces? We shall continue to study this problem in a separate paper.
\end{remark}
\noindent {\bf Acknowledgment:}   The paper is supported partially  by the Nature Science Foundations of  China (12271442, 11661082).

\end{document}